\documentclass{amsart}
\usepackage{amsmath,amsthm,amsfonts,amssymb}
\usepackage{mathrsfs}
\usepackage{latexsym}
\usepackage[latin1]{inputenc}
\usepackage{bm}
\usepackage{amssymb}
\usepackage[all]{xy}
\usepackage{enumerate}
\usepackage{stmaryrd}
\usepackage{color}
\usepackage{mathrsfs}
\usepackage{url}
\usepackage{soul}
\usepackage{ dsfont }
\usepackage{cite}
\usepackage{hyperref}
\usepackage{tikz}
\usetikzlibrary{babel,arrows,positioning,chains,matrix,scopes,cd,quotes,calc,decorations.pathmorphing}

\tikzset{curve/.style={settings={#1},to path={(\tikztostart)
    .. controls ($(\tikztostart)!\pv{pos}!(\tikztotarget)!\pv{height}!270:(\tikztotarget)$)
    and ($(\tikztostart)!1-\pv{pos}!(\tikztotarget)!\pv{height}!270:(\tikztotarget)$)
    .. (\tikztotarget)\tikztonodes}},
    settings/.code={\tikzset{quiver/.cd,#1}
        \def\pv##1{\pgfkeysvalueof{/tikz/quiver/##1}}},
    quiver/.cd,pos/.initial=0.35,height/.initial=0}

\tikzset{tail reversed/.code={\pgfsetarrowsstart{tikzcd to}}}
\tikzset{2tail/.code={\pgfsetarrowsstart{Implies[reversed]}}}
\tikzset{2tail reversed/.code={\pgfsetarrowsstart{Implies}}}
\tikzset{no body/.style={/tikz/dash pattern=on 0 off 1mm}}

\usepackage{pagecolor}

\DeclareMathOperator{\menos}{\setminus}
\newcommand{\fecho}{\overline}

\newtheorem{prop}{Proposition}[section]

\newtheorem{cor}[prop]{Corollary}

\newtheorem{prob}[prop]{Problem}
\newtheorem{remark}[prop]{Remark}
\newtheorem{example}[prop]{Example}
\begin{document}
\newenvironment{sketch}{\paragraph{Sketch:}}{\hfill$\square$}

\title{Selective game version of q-points} 

\author[D. A. Lara]{Dione A. Lara}
\address{Instituto de Ci\^encia Tecnologia e Inova\c{c}\~ao,
Universidade Federal de Lavras, Campus Para\'iso,
S\~ao Sebasti\~ao do Para\'iso, MG, 37950-000, Brazil}
\email{dione.lara@ufla.br}

\author[R. M. Mezabarba]{Renan M. Mezabarba}
\address{Centro de Ci\^encias Exatas,
Universidade Federal do Esp\'irito Santo,
Vit\'oria, ES, 29075-910, Brazil}
\email{renan.mezabarba@ufes.br}

\keywords{topological games, selection principles, $q$-points, local properties}

\subjclass[2020]{54A20, 91A44, 54E18.}
\maketitle

\begin{abstract}
This work presents the selection principle $S_1^*(\tau_x,CD)$ that characterizes $q$-points. We also discuss the induced topological game $G_1^*(\tau_x,CD)$ and its relations with $W$-points and $\widetilde{W}$-points, as well as with the game $G_1(\Omega_x,\Omega_x)$.
\end{abstract}

\section*{Introduction}

First countable spaces and their many generalizations have a long history in General Topology. Among the numerous properties introduced to generalize these spaces, some of them can be viewed as selective properties concerning limit points or convergent sequences. In this work we shall analyze some of these properties in the framework of selection principles, as defined by Scheepers~\cite{scheepers.1}.

According to Michael~\cite{michael}, a point $x$ of a topological space $X$ is a \textbf{$q$-point} if it has a sequence of neighborhoods $(V_n)_{n\in\omega}$ such that, if $(x_n)_{n\in\omega}$ is an injective sequence of points with $x_n\in V_n$ for all $n$, then the set $\{x_n:n\in\omega\}$ has an accumulation point in $X$. The space is called a \textbf{$q$-space} if every point is a $q$-point.

The definition of $W$ and $w$-points, due to Gruenhage~\cite{gruenhage}, depends on a two-person infinite game, played as follows: for a fixed point $x$ of a topological space $X$, Player I chooses an open set $U_0$ with $x\in U_0$, which Player II answers with a point $x_0\in U_0$; Player I then picks an open set $U_1$ containing the point $x$, and Player II chooses some point $x_1\in U_1$, and so on; Player I wins a play of this game if the sequence of points selected by Player II converges to $x$. The point $x$ is a \textbf{$W$-point} if Player I has a winning strategy in this game, and it is called a $w$-point if Player II does not have a winning strategy. We say $X$ is a \textbf{$W$-space} if every point of $X$ is a $W$-point, with a similar terminology for \textbf{$w$-spaces}.

Finally, in the more recent work of Dolen\v{z}al and Moors~\cite{dolenzal}, $W$-spaces are generalized with an adaptation of the previous game: first, instead of choosing open sets that contain the point $x$, Player I has to select simply nonempty open sets while Player II follows the previous rules; the second difference is the winning condition for Player I, which in this game just asks for the sequence of points selected by Player II to have $x$ as an accumulation point\footnote{This winning condition is not effectively different, since Theorem 3.9 in~\cite{gruenhage} shows that the winning condition in the definition of $W$-points can be replaced by the one used in the definition of $\widetilde{W}$-points.}. Then, the point $x$ is said to be a \textbf{$\widetilde{W}$-point} if Player I has a winning strategy in this game, and the space $X$ is a \textbf{$\widetilde{W}$-space} if every point is a $\widetilde{W}$-point.

At this \emph{point}, the reader which is familiar with selection principles and their game versions, as defined by Scheepers~\cite{scheepers.1}, may have noticed that each one of the above properties are appropriate variations of $S_1(\mathcal{A,B})$ principles and $G_1(\mathcal{A,B})$ games.

We shall denote by $S_1(\mathcal{A,B})$ the assertion that for every sequence $(A_n)_{n\in\omega}$ of members of $\mathcal{A}$ there exists a set $\{b_n:n\in\omega\}\in\mathcal{B}$ where $b_n\in A_n$ for all $n$. Tied to that we have the game $G_1(\mathcal{A,B})$, which 
consists of a two-person infinite game, played as follows: at the first inning Player I begins by choosing an element $A_0\in\mathcal{A}$ and Player II answers with an element $b_0\in A_0$; at the next inning Player I chooses an element $A_1\in\mathcal{A}$ and Player II answers with an element $b_1\in A_1$, and so on; Player I wins a play of this game if the set $\{b_n:n\in\omega\}$ belongs to $\mathcal{B}$.

A \textbf{strategy} for a player is a function, let us say $\sigma$, determining how that player shall answer her opponent based on all their previous choices. We say that a strategy $\sigma$ is a \textbf{winning strategy} for its player if there is no way to her opponent to defeat it in any legal play according to $\sigma$. We write I $\uparrow G_1(\mathcal{A},\mathcal{B})$ in order to abbreviate the assertion ``Player I has a winning strategy'', whose negation we denote by I $\not\,\uparrow G_1(\mathcal{A},\mathcal{B})$. Similar notations are adopted regarding Player II. In general we always have that $\neg S_1(\mathcal{A,B})$ implies I $\uparrow G_1(\mathcal{A},\mathcal{B})$, which in turns implies II $\not\,\uparrow G_1(\mathcal{A,B})$.


This work is organized as follows. In the first section, we analyze the $q$-spaces, $W$-spaces, $w$-spaces and $\widetilde{W}$-spaces under the framework of selection principles, putting all these properties in the same picture. Sections 2 and 3 deal with the variations of $q$-points and $\widetilde{W}$-points, respectively, while in the last section we discuss related problems and possible directions of further investigation.






\section{The big picture}

Along this work, for an infinite T$_1$ space $(X,\tau)$ and a non-isolated point $x$, let us denote by
\begin{itemize}
    \item $\tau_x$ the family of all open sets of $X$ containing $x$,
    \item $\Omega_x=\left\{A\subseteq X:x\in\overline{A}\setminus A\right\}$ and
    \item $\Gamma_x=\{A\in\Omega_x:\forall V\in\tau_x\,|A\setminus V|<\aleph_0\}$.
\end{itemize}

One can easily see that a sequence $(x_n)_{n\in\omega}$ in $X$ converges to $x$ if and only if its image $\{x_n:n\in\omega\}$ belongs to $\Gamma_x$. Putting $\neg\Gamma_x=\{A\subseteq X:A\not\in \Gamma_x\}$, it follows that
\begin{itemize}
    \item $x$ is a $W$-point if and only if Player I has a winning strategy in the game $G_1(\tau_x,\neg\Gamma_x)$, and
    \item $x$ is a $w$-point if and only if Player II does not have a winning strategy in the game $G_1(\tau_x,\neg\Gamma_x)$.
\end{itemize}

The first variation is obtained with (the negation of) the $S_1$-version of the above principle: $\neg S_1(\tau_x,\neg\Gamma_x)$, which automatically implies the $W$-property. As it happens with the original $W$-property, first countability also implies $\neg S_1(\tau_x,\neg\Gamma_x)$. And it is not difficult to show that $\neg S_1(\tau_x,\neg\Gamma_x)$ also implies first countability.



We now turn to the Dolen\v{z}al and Moors'~\cite{dolenzal} variation of $W$-points. Since the definition of $\widetilde{W}$-points concerns accumulation points of a sequence rather than a set, we need to consider a slight variation of the selection principles presented in Introduction: we shall denote by $(S_1)(\mathcal{A,B})$ the assertion that for every sequence $(A_n)_{n\in\omega}$ of members of $\mathcal{A}$ there exists a sequence $(b_n)_{n\in\omega}\in\mathcal{B}$ where $b_n\in A_n$ for all $n$. Thus, the only difference between $S_1(\mathcal{A,B})$ and $(S_1)(\mathcal{A,B})$ is that in the first we ask for $\{b_n:n\in\omega\}\in\mathcal{B}$, while in the second we impose $(b_n)_{n\in\omega}\in \mathcal{B}$. A similar variation is understood for the game $(G_1)(\mathcal{A,B})$.

With this in mind, let $L_x$ be the collection of all sequences accumulating at the point $x$. Now, by taking $\neg L_x$ to be the set of all other sequences, it follows that $x$ is a $\widetilde{W}$-point if and only if Player I does have a winning strategy in the game $(G_1)(\tau^*,\neg L_x)$, where we set $\tau^*=\tau\setminus\{\emptyset\}$. This give us two variations of $\widetilde{W}$-points, namely
\begin{itemize}
    \item those points $x$ satisfying $\neg (S_1)(\tau^*,\neg L_x)$ and
    \item those points $x$ such that Player II does not have a winning strategy at the game $(G_1)(\tau^*,\neg L_x)$.
\end{itemize}

Since $\tau_x\subseteq\tau^*$ and every sequence converging to $x$ belongs to $L_x$, it follows that
\begin{align}
 \label{s1(s1)}\neg S_1(\tau_x,\neg\Gamma_x)\Rightarrow&\, \neg (S_1)(\tau^*,\neg L_x)\\
    \label{W.thenWtilde}\text{I}\uparrow G_1(\tau_x,\neg\Gamma_x)\Rightarrow &\, \text{I}\uparrow (G_1)(\tau^*,\neg L_x)
\end{align}
where \eqref{W.thenWtilde} reads as ``every $W$-point is a $\widetilde{W}$-point'', what is already pointed out in \cite{dolenzal}, and  \eqref{s1(s1)} imply $(S_1)(\tau^*,\neg L_x)$ does not hold for points with countable local bases. Regarding local bases, Theorem 3.3 in \cite{gruenhage} states that the image of a winning strategy for Player I in the game $G_1(\tau_x,\neg\Gamma_x)$ provides a local base at $x$. Then it is natural to expect the following.

\begin{prop}
Let $X$ be a topological space. If $\sigma$ is a winning strategy for Player I in the game $(G_1)(\tau^*,\neg L_x)$, then the image of $\sigma$ is a $\pi$-base at the point $x$.\end{prop}

\begin{proof} The argument is similar to the proof of Theorem 3.3 in \cite{gruenhage}. Let $\operatorname{Im}(\sigma)$ be the image of $\sigma$ and suppose that it is not a $\pi$-base for the point $x$. If this is the case, then there exists an open set $U\in\tau^*$ containing $x$ such that every $V\in\operatorname{Im}(\sigma)$ satisfies $V\nsubseteq U$, which gives a way for Player II to win a play. Indeed, it is sufficient that Player II to choose points not belonging to $U$. If $(y_n)_{n\in\omega}$ is the sequence of the choices made by Player II in such a way, then $x$ is not an accumulation point of $(y_n)_{n\in\omega}$, since $y_n\not\in U$ for every $n\in\omega$.
\end{proof} 

With a similar reasoning, one can shows that the condition $\neg (S_1)(\tau^*,\neg L_x)$ is equivalent to the existence of a countable open $\pi$-network at the point $x$. 

\begin{remark} Theorem 3.3 in \cite{gruenhage} also shows that if $X$ is a $W$-space, then
\[\chi(X)=\min\{\operatorname{Im}(\sigma):\sigma\text{ is a winning strategy for Player I in }G_1(\tau_x,\neg\Gamma_x)\}\leq |X|,\]
where $\chi(X)$ denotes the character of $X$. This is also the case for $\widetilde{W}$-spaces if one replaces $\chi$ with $\pi\chi$, the $\pi$-character of $X$. The proof, which is essentially the same as the one presented in \cite{gruenhage}, is left to the reader.\end{remark}

Since $w$-points are those for which Player II does not have a winning strategy in the game $G_1(\tau_x,\neg\Gamma_x)$, one could call by \textbf{$\widetilde{w}$-points} those points satisfying II $\not\,\uparrow (G_1)(\tau^*,\neg L_x)$. However, since a sequence does not converging to $x$ may still have $x$ as an accumulation point, it is not clear what is the relation between $w$-points and $\widetilde{w}$-points.

\begin{prob}
Does every $w$ point is a $\widetilde{w}$-point? Does the converse hold?
\end{prob}

Still regarding $w$ and $\widetilde{w}$-points, Gruenhage \cite{gruenhage} implicitly shows that every $w$-space is \textbf{strictly Fr\'{e}chet} \cite{Gerlits1982}, meaning that $S_1(\Omega_x,\Gamma_x)$ holds for every point $x$ in the space. Indeed, a sequence $(F_n)_{n\in\omega}$ of subsets\footnote{Gruenhage considers a decreasing sequence of subsets containing $x$ and their closure, which corresponds to the countably bi-sequential condition, also known as \textbf{strongly Fr\'{e}chet}\cite{siwiec}.} such that $x\in\overline{F_n}$ for every $n$ induces a strategy for Player II in the game $G_1(\tau_x,\neg \Gamma_x)$ by choosing a point in $A_n\cap F_n$ for every sequence $(A_0,\dotso,A_n)$ of neighborhoods of $x$. Since Player II does not have a winning strategy, there exists a play according to this strategy in which Player II loses, thus giving a sequence $(x_n)_{n\in\omega}$ such that $x_n\to x$ and $x_n\in F_n$ for every $n$. However, the same argument does not hold for $\widetilde{w}$-spaces, motivating the following.

\begin{prob} Does every $\widetilde{w}$-space is (strictly) Fr\'{e}chet?\end{prob}


Finally, let us deal with $q$-points. Once again, since the original definition of Michael~\cite{michael} asks for a different type of selection, we need to consider a slight variation of the selection principle $S_1(\mathcal{A,B})$. We shall denote by $S_1^*(\mathcal{A,B})$ the assertion that for every sequence $(A_n)_{n\in\omega}$ of members of $\mathcal{A}$ there exists an injective sequence $(b_n)_{n\in\omega}$ where $b_n\in A_n$ for all $n$ such that $\{b_n:n\in\omega\}\in\mathcal{B}$. For the game case, $G_1^*(\mathcal{A,B})$ is exactly like $G_1(\mathcal{A,B})$, except for the additional rule asking for Player II to select different elements at each inning of a play.

Then, a point $x\in X$ is a $q$-point if and only if $S_1^*(\tau_x,CD)$ does not hold, where $CD=\{A\subseteq X:A$ is closed and discrete$\}$. The game variations then gives two generalizations of $q$-points
\begin{itemize}
    \item those points $x$ such that I~$\uparrow G_1^*(\tau_x,CD)$ and
    \item those points $x$ such that II~$\not\,\uparrow G_1^*(\tau_x,CD)$.
\end{itemize}

The next diagram summarizes the logical implications between all the properties considered so far, where the original ones are stated with their classic names.\small
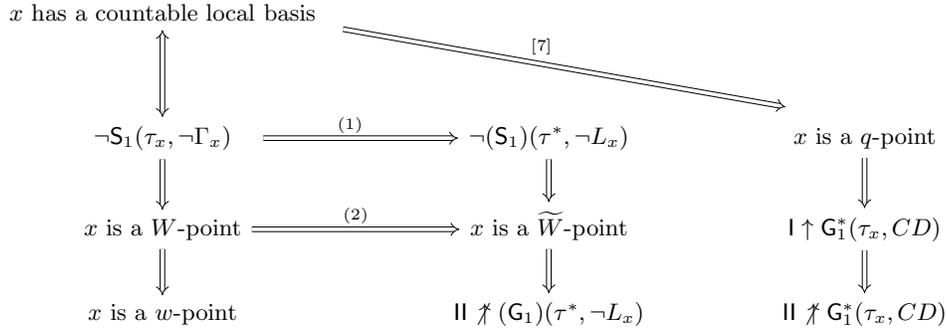
\begin{figure}[h!]
\centering\begin{tikzcd}
	{x~\text{has a countable local basis}} \\
	\\
	{\neg\textsf{S}_1(\tau_x,\neg\Gamma_x)} && {\neg(\textsf{S}_1)(\tau^*,\neg L_x)} && {x~\text{is a }q\text{-point}} \\
	{x~\text{is a }W\text{-point}} && {x~\text{is a }\widetilde{W}\text{-point}} && {\textsf{I}\uparrow \textsf{G}_1^*(\tau_x,CD)} \\
	{x~\text{is a }w\text{-point}} && {\textsf{II}\not\,\uparrow(\textsf{G}_1)(\tau^*,\neg L_x)} && {\textsf{II}\not\,\uparrow\textsf{G}_1^*(\tau_x,CD)}
	\arrow[Rightarrow, from=3-1, to=4-1]
	\arrow[Rightarrow, from=3-3, to=4-3]
	\arrow[Rightarrow, from=3-5, to=4-5]
	\arrow[Rightarrow, from=4-3, to=5-3]
	\arrow[Rightarrow, from=4-1, to=5-1]
	\arrow["\eqref{W.thenWtilde}", Rightarrow, from=4-1, to=4-3]
	\arrow["\eqref{s1(s1)}", shorten <=8pt, Rightarrow, from=3-1, to=3-3]
	\arrow[Leftrightarrow, from=1-1, to=3-1]
	\arrow[Rightarrow, from=4-5, to=5-5]
	\arrow["\text{\cite{michael}}", shift left=3, shorten <=35pt, Rightarrow, from=1-1, to=3-5]
\end{tikzcd}
\caption{The (incomplete) big picture.}
\end{figure}
\normalsize

In the next section we shall see how another weakening of first countability, namely the G$_\delta$-point condition, can be used to relate $q$ points with $\widetilde{W}$-points and their selective variations.

\section{The influence of G$_\delta$-points}

Note that by definition, the open sets played by Player I in the game $G_1^*(\tau_x,CD)$ satisfy the conditions of the game $(G_1)(\tau^*,\neg L_x)$. However, the winning condition for Player I in the former game requires that the points chosen by Player II accumulates at \emph{some point}, while the winning conditions for Player I in the later game asks for the point $x$ to be such an accumulation point. We shall see that the G$_\delta$ condition can be used to force this to happen.

\begin{prop}\label{na:acum}
Let $X$ be a regular space such that $x\in X$ is a $G_{\delta}$-point. If I $\uparrow G_1^*(\tau_x,CD)$, then I $\uparrow (G_1)(\tau^*,\neg L_x)$, i.e., $x$ is a $\widetilde{W}$-point.
\end{prop}
\begin{proof}

Let $\sigma$ be a winning strategy for Player I in $G_1^*(\tau_x,CD)$ and let $\{U_n\}_{n\in\omega}$ be a countable family of open sets such that $\bigcap_{n\in\omega}U_n=\{x\}$. We shall obtain a winning strategy for Player I in the game $(G_1)(\tau^*,\neg L_x)$ in the following way.

Player I starts with $V_0=\sigma(\emptyset)\cap U_0$, to which Player II responds by picking a point $x_0\in V_0$. Since $X$ is regular, there is an open set $A_0\in\tau_x$ such that $x\in A_0\subset\fecho{A_0}\subset V_0$ with $x_0\notin\fecho{A_0}$. At the next inning, Player I chooses $V_1=\sigma(x_0)\cap A_0\cap U_1$, to which Player II replies with a point $x_1\in V_1$, and again we may take an open set $A_1\in\tau_x$ with $x\in A_1\subset\fecho{A_1}\subset V_1$ and $x_1\notin\fecho{A_1}$. 

At the $n$-th inning, the regularity of $X$ allows us to take an open set $A_{n-1}$ such that $x\in A_{n-1}\subset\fecho{A_{n-1}}\subset V_{n-1}$ with $x_{n-1}\notin\fecho{A_{n-1}}$. So, following the previous pattern, Player I chooses $V_n=\sigma(x_0,...,x_{n-1})\cap A_{n-1}\cap U_n$. This describe a strategy for Player I in the game $G_1(\tau^*,\neg L_x)$.

Now, let $P=(V_0,x_0,V_1,x_1,\dotso)$ be a play in this game accordingly to the strategy described above. Notice that:
\begin{itemize}
\item $\bigcap_{n\in\omega}V_n=\{x\}$ and for every $n\in\omega$ we have $\fecho{V_{n+1}}\subset V_n$, since $V_{n+1}=\sigma(x_{n+1})\cap A_n\cap U_{n+1}\subset\fecho{A_{n+1}}\subset V_n$;
\item for every $n\in\omega$ we have $x_n\in V_n\menos\fecho{A_n}$ with $x_k\notin V_n\menos\fecho{A_n}$ for all $k>n$; since $X$ is a T$_1$ space, it follows that $x_n$ is not an accumulation point of the sequence.
\end{itemize}

By the way Player I chooses their open sets, the points selected by Player II are pairwise distinct, then corresponding to legal choices in the game $G_1^*(\tau_x,CD)$. Since $\sigma$ is a winning strategy for Player I in the later game, it follows that $(x_n)_{n\in\omega}$ accumulates in $X$, and the remarks above guarantee that the point has to be $x$. Indeed, for $y\in X$ such that $y\not\in\{x_n:n\in\omega\}\cup\{x\}$, there exists $k\in\omega$ such that $y\notin\fecho{V_k}$, implying that $X\setminus\overline{V_k}$ is an open set containing $y$ but only finitely many points of the sequence. Thus the sequence $(x_n)_{n\in\omega}$ can only accumulate at $x$.
\end{proof}

\begin{cor} Every G$_\delta$ $q$-point of a regular space is also a $\widetilde{W}$-point.\end{cor}

Another way to obtain the previous corollary is through the corresponding implication regarding the $S_1$-version of the previous proposition.

\begin{prop}\label{s1.onion}
Let $X$ be a regular space such that $x$ is a G$_\delta$-point. If $x$ is a $q$-point, then $(S_1)(\tau^*,\neg L_x)$ does not hold.
\end{prop}
\begin{proof}
Let $\{U_n\}_{n\in\omega}$ be a countable family of open sets such that $\{x\}=\bigcap_{n\in\omega}U_n$ and let $(W_n)_{n\in\omega}$ be a sequence of open sets witnessing the $q$-point condition for $x$. Proceeding like in the previous proposition, one may use the regularity to obtain a sequence $(V_n)_{n\in\omega}$ of open sets witnessing the failure of $(S_1)(\tau^*,\neg L_x)$. The details are left to reader.
\end{proof}
\begin{example}\label{example.omega1.1} 
Without the G$_\delta$-point assumption, the previous proposition may fail in general. For instance, let $X$ be the space $\omega_1+1$ and consider the point $\omega_1\in X$.

Notice that the following defines a winning strategy for Player I in the game $G_1^*(\tau_{\omega_1},CD)$: Player I starts with $V_0=X$, to which Player II responds with some point $x_0$; at the next inning, Player I selects $V_1=\left.\left]x_0+1,\omega_1\right.\right]$ and so on\footnote{Since the rules of the game $G_1^*(\tau_x,CD)$ prevents Player II of choosing the point $\omega_1$ twice, we may suppose Player II do not pick this point.}; since the points picked by Player II are increasing, the set $\{x_n:n\in\omega\}$ accumulates at $\sup_{n\in\omega}x_n<\omega_1$.

On the other hand, Player I does not have a winning strategy in the game $(G_1)(\tau^*,\neg L_{\omega_1})$: as implicitly showed in the above paragraph, it is enough for Player II to always choose points different from $\omega_1$.
\end{example}

The pattern becomes complete once we show that with the G$_\delta$-point assumption, II $\not\,\uparrow G_1^*(\tau_x,CD)$ also implies II $\not\,\uparrow(G_1)(\tau^*,\neg L_x)$.

\begin{prop}\label{II.onion} Let $X$ be a regular space with a G$_\delta$-point $x$. If II $\not\,\uparrow G_1^*(\tau_x,CD)$, then II $\not\,\uparrow(G_1)(\tau^*,\neg L_x)$.\end{prop}
\begin{proof}
Let $\mu$ be a strategy for Player II in the game $(G_1)(\tau^*,\neg L_x)$. Since $\mu$ knows how to answer to every nonempty open set of $X$, we can use it to define a strategy for Player II in the game $G_1^*(\tau_x,CD)$, where the hypothesis shall give a play in which Player II loses, meaning that the points selected along the innings accumulate at some point. The G$_\delta$ condition will guarantee that this point is $x$. Let $\{U_n\}_{n\in\omega}$ be a countable family of open sets such that $\{x\}=\bigcap_{n\in\omega}U_n$.

If Player I starts with $V_0\in\tau_x$, let Player II responds with $x_0=\mu(V_0\cap U_0)$. In the next inning, if Player I chooses an open set $V_1\in\tau_x$, the regularity of $X$ gives an open set $A_0\in\tau_x$ such that $\overline{A_0}\subseteq V_0$ and $x_0\not\in \overline{A_0}$, which we use to define $x_1$ as $\mu(V_0\cap U_0, V_1\cap A_0\cap U_1)$. Proceeding like this, we obtain a strategy for Player II in the game $G_1^*(\tau_x,CD)$. Similarly as in the previous propositions, a play in this game lost by Player II induces a sequence which accumulates at $x$, showing that $\mu$ is a not a winning strategy.
%
%
%
\end{proof}

The next diagram summarizes the results of the current section.
\begin{figure}[h!]
\centering
\begin{tikzcd}
	{\neg(\textsf{S}_1)(\tau^*,\neg L_x)} && {x~\text{is a }q\text{-point}} \\
	{x~\text{is a }\widetilde{W}\text{-point}} && {\textsf{I}\uparrow \textsf{G}_1^*(\tau_x,CD)} \\
	{\textsf{II}\not\,\uparrow(\textsf{G}_1)(\tau^*,\neg L_x)} && {\textsf{II}\not\,\uparrow\textsf{G}_1^*(\tau_x,CD)}
	\arrow[Rightarrow, from=1-1, to=2-1]
	\arrow[Rightarrow, from=1-3, to=2-3]
	\arrow[Rightarrow, from=2-1, to=3-1]
	\arrow[Rightarrow, from=2-3, to=3-3]
	\arrow[Rightarrow, from=1-3, to=1-1]
	\arrow[Rightarrow, from=2-3, to=2-1]
	\arrow[Rightarrow, from=3-3, to=3-1]
\end{tikzcd}
\caption{The horizontal implications, under the hypotheses that $X$ is regular with $x$ a G$_\delta$-point.}
\end{figure}

\section{Duality and countable strong fan tightness}

Recall that a topological space has \textbf{countable strong fan tightness} at a point $x\in X$~\cite{Sakai1988} if $S_1(\Omega_x,\Omega_x)$ holds. Since every point with a countable local basis has countable strong fan tightness, this property can be viewed as an intermediate between first countability and countable tightness.

Following the terminology of \cite{alster} we say that two games $G$ and $G'$ are dual if
\begin{itemize}
    \item Player I has a winning strategy in $G$ if and only if Player II has a winning strategy in $G'$ ; and
    \item Player II has a winning strategy in $G$ if and only if Player I has a winning strategy in $G'$ .
\end{itemize}

\begin{prop}\label{dual}
The games $G_1(\Omega_x,\bigcup_{p\in X}\Omega_p)$ and ${G_1^*(\tau_x,CD)}$ are dual.\end{prop}
\begin{proof}

Let us first analyze how a winning strategy for Player I in one of the games yields a winning strategy for Player II in the other game.\medskip

\noindent (i)$\quad$I $\uparrow G_1(\tau_x,CD)$\,$\Rightarrow$ II $\uparrow G_1(\Omega_x,\bigcup_{p\in X}\Omega_p)$.\medskip

Let $\sigma$ be a winning strategy for Player I in the game $G_1^*(\tau_x,CD)$. Since the choices of Player I in the game $G_1(\Omega_x,\bigcup_{p\in X}\Omega_p)$ intercept every open set in $\tau_x$, one can readily defines a winning strategy for Player II in this game by choosing points in the open sets selected by $\sigma$. The details are left to reader.\medskip

\noindent (ii)$\quad$I $\uparrow G_1(\Omega_x,\bigcup_{p\in X}\Omega_p)$ $\Rightarrow$ II $\uparrow G_1^*(\tau_x,CD)$.\medskip

Let $\rho$ be a winning strategy for Player I in the game $G_1(\Omega_x,\bigcup_{p\in X}\Omega_p)$. If $A_0\in \tau_x$ is the first move of Player I in the game $G_1^*(\tau_x,CD)$, then Player II may select a point $x_0$ belonging to $\rho(\emptyset)\cap A_0$, what can be done since $\rho(\emptyset)\in\Omega_x$. If Player I responds with $A_1\in\tau_x$, then again Player II may select $x_1\in\rho(x_0)\cap (A_1\setminus\{x_0\})$, and so on. Since the strategy $\rho$ is winning, it follows that by the end of a play $(A_0,x_0,A_1,x_1,\dotso)$ we have $\{x_n:n\in\omega\}\not\in\Omega_p$ for all $p\in X$, i.e., $\{x_n\}_{n\in\omega}\in CD$.\medskip


Now we shall see how winning strategies for Player II in one of the games give winning strategies for Player I in the other game.\medskip

\noindent (iii)$\quad$II $\uparrow G_1^*(\tau_x,CD)$ $\Rightarrow$ I $\uparrow G_1(\Omega_x,\bigcup_{p\in X}\Omega_p)$.\medskip

Let $\sigma$ be a winning strategy for the Player II in the game $G_1^*(\tau_x,CD)$. We first show that $\{\sigma(V):V\in\tau_x\}\in\Omega_x$. If this is not the case, then there is a $U\in\tau_x$ such that $U\cap\{\sigma(V):V\in\tau_x\}=\emptyset$, which is absurd since $\sigma(U)\in U$. 
A similar argument shows that $\{\sigma(V_0,...,V_n,V):V\in\tau_x\}\in\Omega_x$ for every $V_0,...,V_n\in\tau_x$.
Thus Player I may use the strategy $\sigma$ to choose subsets in $\Omega_x$ while keeping track of a valid play on game $G_1(\tau_x,NA)$: Player I starts with $A_0=\{\sigma(V):V\in\tau_x\}$, then responds to a Player II's choice, say $\sigma(V_0)$, with $A_1=\{\sigma(V_0,V):V\in\tau_x\}$ and so on. It is clear that Player I wins every play of $G_1(\Omega_x,\bigcup_{p\in X}\Omega_p)$ with this strategy.\medskip 

\noindent (iv)$\quad$II $\uparrow G_1(\Omega_x,\bigcup_{p\in X}\Omega_p)$ $\Rightarrow$ I $\uparrow G_1^*(\tau_x,CD)$.\medskip 

Let $\rho$ be a winning strategy for Player II in the game $G_1(\Omega_x,\bigcup_{x\in X}\Omega_x)$. First of all, we note that there is an open set $V_0\in\tau_x$ such that each point $y\in V_0$ is the first movement of Player II with respect to $\rho$, i.e., there is an $A\in\Omega_x$ such that $y=\rho(A)$. If this is not the case, then we may obtain a subset $C\in\Omega_x$ such that $\rho(C)\not\in C$, which is absurd. As in the previous paragraph, Player I may use this neighborhood $V_0$ as her first movement, to which Player II responds with a point $x_0=\sigma(A_0)$ for some $A_0\in \Omega_x$. Proceeding like this, it is easy to see that Player I obtains a winning strategy in the game $G_1^*(\tau_x,CD)$, as desired.
\end{proof}

Since $\Omega_x\subseteq \bigcup_{p\in X}\Omega_p$, both the implications II $\uparrow G_1(\Omega_x,\Omega_x)\Rightarrow$ {I $\uparrow G_1^*(\tau_x,CD)$} and II $\uparrow G_1^*(\tau_x,CD)\Rightarrow$ I $\uparrow G_1(\Omega_x,\Omega_x)$ hold, and none of these are reversible, as the space $X=\omega_1+1$ shows: as we already showed in Example~\ref{example.omega1.1}, $X$ satisfies I $\uparrow G_1^*(\tau_x,CD)$, and Player I can win every play of the game $G_1(\Omega_{\omega_1},\Omega_{\omega_1})$ by choosing the subset $[0,\omega_1)$ at every inning. Once again, the G$_\delta$-condition gives one of the converses.





\begin{prop}\label{na:oo}
Let $X$ be a regular space and let $x\in X$ be a $G_{\delta}$-point. If I $\uparrow G_1^*(\tau_x,CD)$ then II $\uparrow G_1(\Omega_{x},\Omega_{x})$.\end{prop}
\begin{proof} As in the proof of Proposition \ref{na:acum}, let us to take a winning strategy $\sigma$ for Player I in the game $(G_1)(\tau^*,\neg L_x)$ such that $x\in\bigcap \operatorname{im}(\sigma)$. Now, Player II may use $\sigma$ to play in the game $G_1(\Omega_{x},\Omega_{x})$ as follows: if $A_0\in\Omega_{x}$ is the first move of Player I in the game $G_1(\Omega_{x},\Omega_{x})$, then Player II picks a point $x_0\in \sigma(\emptyset)\cap A_0$, what can be done since $\sigma(\emptyset)\in \tau_{x}$; at the next inning, Player I chooses $A_1\in\Omega_{x}$ and Player II answers with $x_1\in \sigma(x_0)\cap A_1$. Proceeding like this we obtain a winning strategy for Player II in the game $G_1(\Omega_{x},\Omega_{x})$, since a play $(A_0,x_0,A_1,x_1,\dotso)$ in this game, played according with the previous strategy, corresponds to the play $(\sigma(\emptyset),x_0,\sigma(x_0),x_1,\sigma(x_0,x_1),x_2,\dotso)$ in the game $(G_1)(\tau^*,\neg L_x)$ according with the winning strategy $\sigma$, from which it follows that $x\in\overline{\{x_n:n\in\omega\}}$.
%
\end{proof}

\begin{cor}
Every G$_\delta$ $q$-point of a regular space has countable strong fan tightness.
\end{cor}

\begin{prob}
Let $X$ be a regular space and let $x\in X$ be a $G_\delta$-point. Are the games $G_1(\tau_x,CD)$ and $G_1(\Omega_x,\Omega_x)$ dual?\medskip
\end{prob}

Notice that by the previous proposition, the above problem depends on the converse of the implication II $\uparrow G_1^*(\tau_x,CD)\Rightarrow$ I $\uparrow G_1(\Omega_x,\Omega_x)$.

\section{Further questions and comments}

In the next diagram we summarize the results obtained so far. The superscript ``$^*$'' indicates that the implication holds under the hypotheses that $X$ is regular and $x$ is a G$_\delta$-point, while $\neg CD$ abbreviates the set $\bigcup_{p\in X}\Omega_p$.

\tiny
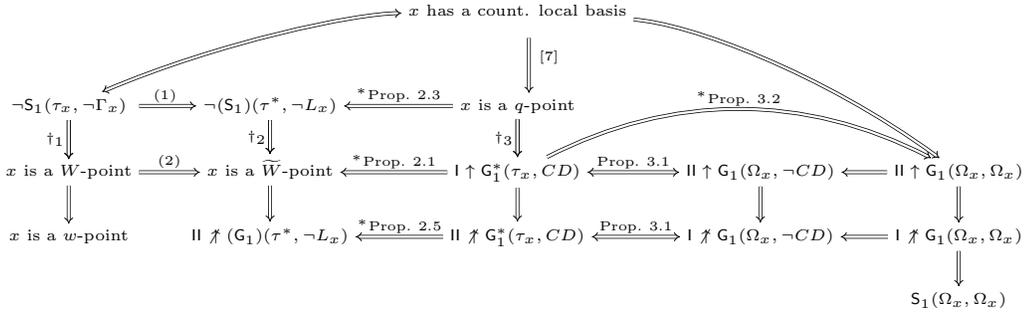
\begin{figure}[h!]
\centering
\begin{tikzcd}
	&& {x~\text{has a count. local basis}} \\
	\\
	{\neg\textsf{S}_1(\tau_x,\neg\Gamma_x)} & {\neg(\textsf{S}_1)(\tau^*,\neg L_x)} & {x~\text{is a }q\text{-point}} \\
	{x~\text{is a }W\text{-point}} & {x~\text{is a }\widetilde{W}\text{-point}} & {\textsf{I}\uparrow \textsf{G}_1^*(\tau_x,CD)} & {\textsf{II}\uparrow\textsf{G}_1(\Omega_x,\neg CD)} & {\textsf{II}\uparrow\textsf{G}_1(\Omega_x,\Omega_x)} \\
	{x~\text{is a }w\text{-point}} & {\textsf{II}\not\,\uparrow(\textsf{G}_1)(\tau^*,\neg L_x)} & {\textsf{II}\not\,\uparrow\textsf{G}_1^*(\tau_x,CD)} & {\textsf{I}\not\,\uparrow\textsf{G}_1(\Omega_x,\neg CD)} & {\textsf{I}\not\,\uparrow\textsf{G}_1(\Omega_x,\Omega_x)} \\
	&&&& {\textsf{S}_1(\Omega_x,\Omega_x)}
	\arrow["\dagger_1"', Rightarrow, from=3-1, to=4-1]
	\arrow["\dagger_2"', Rightarrow, from=3-2, to=4-2]
	\arrow["\dagger_3"', Rightarrow, from=3-3, to=4-3]
	\arrow[Rightarrow, from=3-1, to=4-1]
	\arrow[Rightarrow, from=3-2, to=4-2]
	\arrow[Rightarrow, from=3-3, to=4-3]
	\arrow[Rightarrow, from=4-2, to=5-2]
	\arrow[Rightarrow, from=4-1, to=5-1]
	\arrow["\eqref{W.thenWtilde}", Rightarrow, from=4-1, to=4-2]
	\arrow["\eqref{s1(s1)}", shorten <=2pt, Rightarrow, from=3-1, to=3-2]
	\arrow[curve={height=12pt}, Leftrightarrow, from=1-3, to=3-1]
	\arrow[Rightarrow, from=4-3, to=5-3]
	\arrow["\text{\cite{michael}}", shift left=3, shorten <=6pt, Rightarrow, from=1-3, to=3-3]
	\arrow["{{{}^*\text{Prop.~\ref{II.onion}}}}"', Rightarrow, from=5-3, to=5-2]
	\arrow["{{}^*\text{Prop.~\ref{na:acum}}}"', Rightarrow, from=4-3, to=4-2]
	\arrow["{{}^*\text{Prop.~\ref{s1.onion}}}"', Rightarrow, from=3-3, to=3-2]
	\arrow["{\text{Prop.~\ref{dual}}}", Rightarrow, 2tail reversed, from=5-3, to=5-4]
	\arrow["{\text{Prop.~\ref{dual}}}", Rightarrow, 2tail reversed, from=4-3, to=4-4]
	\arrow[Rightarrow, from=4-5, to=4-4]
	\arrow[Rightarrow, from=5-5, to=5-4]
	\arrow[Rightarrow, from=4-4, to=5-4]
	\arrow[Rightarrow, from=5-5, to=6-5]
	\arrow["{{}^*\text{Prop.~\ref{na:oo}}}", curve={height=-30pt}, Rightarrow, from=4-3, to=4-5]
	\arrow[Rightarrow, from=4-5, to=5-5]
	\arrow[curve={height=-18pt}, Rightarrow, from=1-3, to=4-5]
\end{tikzcd}
\caption{The current big picture.}\label{bigfigure}
\end{figure}

\normalsize

Once these local properties are established under the selection principles landscape, the typical combinatorics questions apply. For instance, considering the equivalence between the $S_1(\mathcal{A,B})$ principle with I $\not\,\uparrow G_1(\mathcal{A,B})$ when $\mathcal{A}$ and $\mathcal{B}$ are replaced by the family of all open coverings\footnote{Pawlikowski~\cite{Pawlikowski1994}.}, one can ask whether something similar happens in the present context. This is the case for $q$-points under the presence of a countable local $\pi$-basis.

\begin{prob}
Does the converse of any of the $(\dagger_i)$ arrows in Figure~\ref{bigfigure} hold?
\end{prob}


There are natural connections with $C_p$-theory and covering properties. Indeed, for a Tychonoff space $Y$, Theorem 4.4 in \cite{McCoybook} establishes that the conditions
\begin{itemize}
\item $C_p(Y)$ is first countable,
\item $C_p(Y)$ is a $q$-space, and
\item $Y$ is countable
\end{itemize}
are equivalent, thus suggesting the following.
\begin{prob}
Is there any uncountable space $Y$ such that $C_p(Y)$ is $W$-space (or has any other property in Figure~\ref{bigfigure}?
\end{prob}

Since the pseudocharacter of $C_p(Y)$ is the density of the space $Y$~\cite{guthrie_1974}, an uncountable separable space $Y$ such that $C_p(Y)$ satisfies II $\uparrow G_1(\Omega_f,\Omega_f)$ could provide positive answers to some instances of the above problem\footnote{Thus, by a result of Scheepers~\cite{Scheep2015}, $Y$ should be a space such that II $\uparrow G_1(\Omega(Y),\Omega(Y))$ holds, where $\Omega(Y)$ is the family of all $\omega$-coverings of $Y$~\cite{Gerlits1982}.}.







\section*{Acknowledgments}

We would like to thank Leandro Aurichi for bringing the subject of q-points to us, which was the initial inspiration for this work.

\bibliography{ref_1} \nocite{*}

\begin{thebibliography}{10}

\bibitem{alster}
{\sc Aurichi, L.~F., and Dias, R.~R.}
\newblock Topological games and {A}lster spaces.
\newblock {\em Canad. Math. Bull. 57}, 4 (2014), 683--696.

\bibitem{dolenzal}
{\sc Dole\v{z}al, M., and Moors, W.~B.}
\newblock On a certain generalization of {$W$}-spaces.
\newblock {\em Topology Appl. 231\/} (2017), 1--9.

\bibitem{Gerlits1982}
{\sc Gerlits, N., and Nagy, Z.}
\newblock Some properties of $\text{C}(\text{X})$, $\text{I}$.
\newblock {\em Topology and its Applications 14}, 2 (1982), 151--161.

\bibitem{gruenhage}
{\sc Gruenhage, G.}
\newblock Infinite games and generalizations of first-countable spaces.
\newblock {\em General Topology and Appl. 6}, 3 (1976), 339--352.

\bibitem{guthrie_1974}
{\sc Guthrie, J.}
\newblock Ascoli theorems and the pseudocharacter of mapping spaces.
\newblock {\em Bulletin of the Australian Mathematical Society 10}, 3 (1974),
  403--408.

\bibitem{McCoybook}
{\sc M{c}Coy, R.~A., and Ntantu, I.}
\newblock {\em Topological properties of spaces of continuous functions}.
\newblock Lecture Notes in Mathematics. Springer-Verlag, Berlin, 1988.

\bibitem{michael}
{\sc Michael, E.}
\newblock A note on closed maps and compact sets.
\newblock {\em Israel J. Math. 2\/} (1964), 173--176.

\bibitem{Pawlikowski1994}
{\sc Pawlikowski, J.}
\newblock Undetermined sets of point-open games.
\newblock {\em Fundamenta Mathematicae 144}, 3 (1994), 279--285.

\bibitem{Sakai1988}
{\sc Sakai, M.}
\newblock Property $\text{C}''$ and function spaces.
\newblock {\em Proceedings of the American Mathematical Society 104}, 3 (1988),
  917--919.

\bibitem{scheepers.1}
{\sc Scheepers, M.}
\newblock Combinatorics of open covers $\text{I}$: Ramsey theory.
\newblock {\em Topology and its Applications 69}, 1 (1996), 31--62.

\bibitem{Scheep1997}
{\sc Scheepers, M.}
\newblock Combinatorics of open covers $\text{III}$: games,
  $\text{C}_p(\text{X})$.
\newblock {\em Fundamenta Mathematicae 152\/} (1997), 231--254.

\bibitem{Scheep2015}
{\sc Scheepers, M.}
\newblock Remarks on countable tightness.
\newblock {\em Topology and its Applications 161}, 1 (2014), 407--432.

\bibitem{siwiec}
{\sc Siwiec, F.}
\newblock Sequence-covering and countably bi-quotient mappings.
\newblock {\em General Topology and its Applications 1}, 2 (1971), 143--154.

\end{thebibliography}
\bibliographystyle{acm}
\end{document}